\documentclass[a4paper,11pt]{amsart}
\usepackage[
  top=3cm,
  bottom=2cm,
  left=2.5cm,
  right=2.5cm,
  headheight=14pt,
  headsep=1cm,
  footskip=1cm
]{geometry}
\usepackage{amsmath, epsfig, verbatim}
\usepackage{amsfonts}
\usepackage{amsthm}
\usepackage{mathtools}
\usepackage{amsmath}
\usepackage{amssymb} 
\usepackage{graphics,graphicx}
\usepackage{epstopdf}
\usepackage{listings}
\usepackage{xcolor}
\usepackage{tikz}
\usepackage{subcaption}
\usepackage{pgfplots}
\usepackage{caption}
\usepackage{booktabs}
\usepackage{tabularx}
\usepackage{hyperref}
\usepackage{doi}
\usepackage[numbers,sort&compress]{natbib}
\usetikzlibrary{calc}
\allowdisplaybreaks
\definecolor{codegray}{rgb}{0.95,0.95,0.95}
\definecolor{pykeyword}{rgb}{0.13,0.13,1}
\definecolor{pystring}{rgb}{0.58,0,0.82}

\lstdefinestyle{pythonstyle}{
    backgroundcolor=\color{codegray},
    language=Python,
    basicstyle=\ttfamily\small,
    keywordstyle=\color{pykeyword}\bfseries,
    stringstyle=\color{pystring},
    commentstyle=\color{gray},
    showstringspaces=false,
    numbers=left,
    numberstyle=\tiny,
    frame=single,
    breaklines=true,
    tabsize=4,
}
\numberwithin{equation}{section}
\theoremstyle{plain}
\newtheorem{theorem}{Theorem}
\newtheorem{defn}[theorem]{Definition}

\newtheorem{lemma}[theorem]{Lemma}
\newtheorem{prop}[theorem]{Proposition}
\begin{document}
\title{Visibility polynomial of corona of two graphs} 
\author{Tonny K B}
\address{Tonny K B, Department of Mathematics, College of Engineering Trivandrum, Thiruvananthapuram, Kerala, India, 695016.}
\email{tonnykbd@cet.ac.in}
\author{Shikhi M}
\address{Shikhi M, Department of Mathematics, College of Engineering Trivandrum, Thiruvananthapuram, Kerala, India, 695016.}
\email{shikhim@cet.ac.in}
\begin{abstract}
In multiagent systems, effective coordination, coverage, and communication often rely on the concept of visibility between agents or nodes within the system. Graph-theoretically, for any subset $X$ of vertices of a graph $G$, two vertices are said to be $X$-visible if there exists a shortest path between them that contains no vertex of $X$ as an internal vertex. In this paper, we investigate the visibility polynomial associated with the corona product of two graphs. The visibility polynomial encodes the number of mutual-visibility sets of all orders within a graph, and the process of enumerating these sets provides a deeper understanding of their structural properties. We characterize the structure of mutual-visibility sets arising specifically within the corona product. As part of this study, we introduce the notion of $C_Q$-visible sets, defined with respect to a selected subset $Q$ of vertices in a graph $G$. A $C_Q$-visible set is a collection of vertices in $\overline{Q}$ that is not only $Q$-visible, but also individually visible from each vertex in $Q$. Using this concept, we establish several characterizations and properties of mutual-visibility sets within the corona product, thereby providing deeper insights into their structure and behavior.
\end{abstract}

\maketitle
\section{Introduction}
 Let $G(V,E)$ be a simple graph and let $X\subseteq V$.  Two vertices $u,v\in V$ are said to be $X$-visible \cite{Stefano} if there exists a shortest path $P$ from $u$ to $v$ such that the internal vertices of $P$ do not belong to $X$; that is, $V(P)\cap X \subseteq \{u, v\}$. A set $X$ is called a mutual-visibility set of $G$ if every pair of vertices in $X$ is $X$-visible. The maximum size of such a set in $G$ is referred to as the mutual-visibility number, denoted by $\mu(G)$. A vertex $u\in V(G)\setminus X$ is said to be $X$-visibile if $u,v$ are $X$-visible for every $v\in X$. 

The notion of mutual visibility in graphs has received increasing attention because of its significance across both theoretical and practical domains. Wu and Rosenfeld first explored the foundational visibility problems in the context of pebble graphs~\cite{Geo_convex_1, Geo_convex_2}. Later, Di Stefano introduced the formal definition of mutual-visibility sets in graph-theoretic terms~\cite{Stefano}. The concept of mutual-visibility serves as a powerful tool for analyzing the transmission of information, influence, or coordination under topological constraints. This graph-theoretic framework has been the subject of numerous investigations~\cite{MV_1, MV_2, MV_3, MV_4, MV_5, MV_6, MV_7, MV_8, MV_9, TMV_1, TMV_2}, and several variants of mutual visibility have also been proposed in ~\cite{MV_10}. 

An associated notion is that of a general position set, which was independently proposed in \cite{GP_1, GP_3}. This concept describes a subset of vertices in a graph where no three distinct vertices lie along the same shortest path. Formally, a set 
$S\subseteq V(G)$ in a connected graph 
$G$ becomes a general position set if none of its vertices appear on a geodesic between any two others in the set. The idea has been extended by defining the general position polynomial of a graph, as introduced and examined in \cite{GP_2}.

In practical applications, mutual visibility is significant in robotics. In multi-agent systems, robots must often reposition themselves to ensure unobstructed visibility between all agents, facilitating decentralized control algorithms for formation, navigation, and surveillance in unknown or dynamic environments. For detailed studies, see~\cite{robotics1,robotics2,robotics3,robotics4,robotics5,robotics6,robotics7}. In robotic systems and sensor networks, effective coordination, coverage, and communication often rely on the concept of visibility between agents or nodes in the system.

Studying mutual visibility sets of all orders enables a deeper understanding of how groups of agents can observe each other simultaneously, which is critical for tasks such as surveillance, target tracking, and distributed decision-making. In environments with varying visibility constraints due to obstacles, analyzing mutual-visibility sets of all sizes helps to determine how many agents can simultaneously maintain line-of-sight and adjust their positions accordingly. In an important development in this direction, B. Csilla et.al. introduced the visibility polynomial in~\cite{sandi}, which serves as a polynomial invariant capturing the distribution of mutual-visibility sets of all orders within a graph. Let $G$ be a graph of order $n$. The visibility polynomial of $G$, denoted by $\mathcal{V}(G)$, is defined in \cite{sandi} as
\[
   \mathcal{V}(G) = \sum_{i \geq 0} r_i x^{i},
\]
where $r_i$ denotes the number of mutual-visibility sets of $G$ having cardinality $i$.
 The visibility polynomial encodes the number of mutual-visibility sets of all orders within a graph, and the process of enumerating these sets provides a deeper understanding of their structural properties. In \cite{VP_1}, the present authors investigated the visibility polynomial associated with the join of two graphs.

This paper investigates the visibility polynomial associated with the corona product of two graphs. As part of this investigation, we analyze the structure of mutual-visibility sets arising within the corona product. To facilitate this analysis, we introduce the concept of $C_Q$-visible sets in Section~\ref{P2.sec3}. A $C_Q$-visible set is a subset $W$ of $\overline{Q}$, where $Q \subseteq V(G)$, that is $Q$-visible and $\{u, w\}$ is $Q$-visible for all $u \in Q$ and $w \in W$. Using this concept, we establish several characterizations and properties of visibility sets within corona products, providing deeper insights into their structure and behavior. In  Section~\ref{sec4}, we introduce absolute-clear graphs, a notion essential for the development of visibility polynomials of corona products.
\section{Notations and preliminaries}
 $G(V,E)$ represents an undirected simple graph with vertex set $V(G)$ and edge set $E(G)$. Unless otherwise stated, all graphs in this paper are assumed to be connected, so that there is at least one path between each pair of vertices. We follow the standard graph-theoretic definitions and notations as presented in \cite{Harary}.

 The complement of a graph $G$ is denoted by $\overline{G}$, which has the same set of vertices as that of $G$ and two vertices in $\overline{G}$ are adjacent if and only if they are not adjacent in $G$. A complete graph $K_n$ on $n$ vertices is a graph in which there is an edge between any pair of distinct vertices. A sequence of vertices $(u_0,u_1,u_2,\ldots,u_{n})$ is referred to as a $(u_0,u_n)$-path in a graph $G$ if $u_iu_{i+1}\in E(G)$, $\forall i\in\{0,1,\ldots,(n-1)\}$. A cycle (or circuit) in a graph $G$ is a path $(u_0,u_1,u_2,\ldots, u_n)$ together with an edge $u_0u_n$. If a graph $G$ on $n$ vertices itself is a path, it is denoted by $P_n$, and if the graph $G$ itself is a cycle, it is denoted by $C_n$.

 The shortest path between two vertices of a graph is called a geodesic. The distance $ d_G(u,v)$ between two vertices $ u$  and $ v$  in $G$ is defined as the length of a geodesic from $u$ to $v$ in $ G$. The maximum distance between any pair of vertices of $G$  is called the diameter of  $G$, denoted by $diam(G)$. Let $X$ be any set. The cardinality of $X$ is denoted as $|X|$. A proper subset of $X$ is a set containing some but not all of the elements of $X$. Let $X\subseteq V(G)$. Then the induced subgraph of $G$ by $X$ is the graph $G[X]$ with vertex set $X$ and with the edges of $G$ having both endpoints in $X$. The diameter of $G[X]$ in the graph $G$ is defined as $diam_G(X) = \displaystyle\max_{u, v \in X} d_G(u, v)$  and is denoted by  $diam_G(X)$.  The number of mutual-visibility sets of order $k$ of $G$  having diameter $d$ in $G$ is denoted by $\Theta_{k,d}(G)$ \cite{VP_1}.
 
  The graph $G\backslash v$ denotes the graph derived from $G$ by removing a single vertex $v$ from $G$ together with all the edges incident on $v$. If the removal of a vertex $v$ disconnects the graph into two or more components, then $v$ is referred to as a cut vertex. 
 
Let $G$ and $H$ be two graphs. The corona of $G$ and $H$, denoted by $G \odot H$, is the graph obtained by taking one copy of $G$ and $|V(G)|$ copies of $H$, and for each vertex $v$ in $G$, joining $v$ to each vertex in the corresponding copy of $H$ associated with  $v$, denoted by $H_v$. 
\section{Mutual-visibility sets of the corona of two graphs }\label{P2.sec3}
In this section, we aim to characterize mutual visibility-sets within the corona product of two graphs.
\begin{lemma}\label{P2.lem1}
    Let $A \subseteq V(G)$. Then $A$ is a mutual-visibility set of $G \odot H$ if and only if it is a mutual-visibility set of $G$.
\end{lemma}
\begin{proof}
  The result follows from the fact that the shortest path between any two vertices of $G$ in $G\odot H$ contains no vertex of $\cup_{w\in V(G)} H_w$.
\end{proof}
\begin{lemma}\label{P2.lem2}
 Let $G$ and $H$ be two graphs. If $S \subseteq \cup_{w \in V(G)} V(H_w) $, then $S$ is a mutual-visibility set of $G \odot H$. 
\end{lemma}
\begin{proof}
Let $a, b \in S$. We consider two cases. First, suppose $a, b \in V(H_v)$ for some $v \in V(G)$. If $a$ and $b$ are not adjacent in $H_v$, then there exist a shortest path $P = (a, v, b)$ from $a$ to $b$ in $G \odot H$ such that $V(P) \cap S = \{a, b\}$. Therefore, $a$ and $b$ are $S$-visible.  

Next, suppose $a$ and $b$ belong to different copies of $H$, say $a \in V(H_{v_1})$ and $b \in V(H_{v_2})$ where $v_1,v_2\in V(G)$. Let $Q$ be a shortest path from $v_1$ to $v_2$ in $G$. Then a shortest path from $a$ to $b$ in $G \odot H$ is obtained by concatenating the paths $P_1 = (a, v_1)$, $Q$, and $P_2 = (v_2, b)$. In this case, $V(P) \cap S = \{a, b\}$, and hence $a$ and $b$ are $S$-visible. Therefore, $S$ is a mutual-visibility set of $G \odot H$.
\end{proof}
\begin{lemma}\label{P2.lem3}
   If $v \in V(G)$ and $\emptyset \subsetneq B\subseteq \cup_{w \in V(G) } V(H_w)$, then the set $\{ v\} \cup B$ is a mutual-visibility set of $G \odot H$ if and only if $B$ satisfies one of the following conditions.\\
    1. $B$ is a mutual-visibility set of $H$ with $diam_{H}(B) \leq 2$.\\
    2. $B \subseteq \cup_{w \in V(G)\setminus \{v\} } V(H_w)$ and $\{v\}$-visible subset of $G \odot H$.    
\end{lemma}
\begin{proof}
Suppose that $S = \{v\} \cup B$ is a mutual-visibility set of $G \odot H$. Then, either $B \subseteq V(H_v)$ or $B \subseteq \cup_{w \in V(G)\setminus \{v\}} V(H_w)$, since every shortest path from $a \in V(H_v)$ to $b \in \cup_{w \in V(G)\setminus \{v\}} V(H_w)$ passes through $v \in S$, which contradicts the mutual-visibility of $S$.

\medskip
\noindent
\textbf{Case 1}: If $B \subseteq V(H_v)$, then every shortest path between any two vertices of $H_v$ in $G \odot H$ lies entirely within the induced subgraph $(G \odot H)[\{v\}\cup V(H_v)]$. Therefore, two vertices $a,b\in B$ are $S$-visible in $G \odot H$ if and only if they are  $B$-visible in $H_v$ and $d_{H_v}(a,b)\le 2$. Indeed, if $d_{H_v}(a,b)\ge 3$, then the unique geodesic between $a$ and $b$ in $G \odot H$ is $(a,v,b)$, whose internal vertex $v$ lies in $S$, contradicting mutual-visibility. Moreover, $v$ and $b \in B$ are $S$-visible, since they are adjacent in $G \odot H$.  Hence, in this case, $\{v\}\cup B$ is a mutual-visibility set of $G\odot H$ if and only if $B$ is a mutual-visibility set of $H_v$ with $\operatorname{diam}_{H_v}(B)\le 2$.

\medskip
\noindent \textbf{Case 2}: Assume that $B \subseteq \cup_{w \in V(G)\setminus \{v\}} V(H_w)$ and $S = \{v\} \cup B$ is a mutual-visibility set of $G \odot H$. Let $a, b \in B$. Then there exist a shortest $(a, b)$-path $P$ such that $V(P) \cap S \subseteq \{a, b\}$. Therefore, $a$ and $b$ are $\{v\}$-visible, hence $B$ is a $\{v\}$-visible subset of $G \odot H$.

Conversely assume that $B \subseteq \cup_{w \in V(G)\setminus \{v\} } V(H_w)$ and $\{v\}$-visible subset of $G \odot H$. Let $a, b \in B$ be non-adjacent vertices. Either both $a$ and $b$ belong to the same copy $H_g$ where $g\in V(G)$, and $ g\neq v$, or to two different copies. In the first case, $P = (a,g,b)$, and $g \neq v$, is a shortest $(a,b)$-path, so $a$ and $b$ are $S$-visible. In the second case, suppose $a \in V(H_{v_1})$ and $b \in V(H_{v_2})$ where $v_1,v_2\in V(G)\setminus \{v\}$. We claim that every shortest $(a,b)$-path has all its internal vertices in $V(G)$. Let $P'$ be a shortest path which enters some copy $H_t$ at a vertex $t \in V(G)$ and visits a vertex $z \in V(H_t)$. It must then leave again through $t$. The path obtained from $P'$ by contracting the subpath $(t,\ldots,t)$ to the single vertex $t$ has length at least two less than $P'$, which contradicts the minimality of $P'$. Hence the claim. By the $\{v\}$-visibility of $B$ there exists a shortest $(a,b)$-path $P$ in $G \odot H$ that avoids $v$. Hence the internal vertices of $P$ lie in $V(G)\setminus\{v\} = V(G)\setminus S$, so $P$ avoids internal vertices of $S$. Therefore, $a$ and $b$ are $S$-visible. Moreover, any shortest $(v,b)$-path $P$, where $b \in B$, satisfies $V(P) \cap S = \{v,b\}$. Therefore, the set $\{v\} \cup B$ is a mutual-visibility set of $G \odot H$. This completes the proof.
\end{proof}

\begin{lemma}\label{P2.lem5}
  Let $A \subseteq V(G)$ where $|A| \geq 2$ and let $\ B\subseteq \cup_{w \in V(G) } V(H_w)$. If $B$ contains at least one vertex of $\cup_{a \in A } H_a$, then $A \cup B$ is not a mutual-visibility set of $G \odot H$.
\end{lemma}
\begin{proof}
Let $b \in B \cap \left( \cup_{a \in A} V(H_a) \right)$. Then $b \in V(H_g)$ for some $g \in A$. Since $|A| \geq 2$, there exists a vertex $c \in A$ other than $g$. Let $P$ be a shortest $(b,c)$-path in $G \odot H$. Then $P$ passes through $g$, and hence $V(P) \cap A \supseteq \{c,g,b\}$.
Therefore, $V(P) \cap (A \cup B) \supseteq \{c, g, b\}$, and hence $A \cup B$ is not a mutual-visibility set of $G \odot H$.
\end{proof}
\begin{figure}[h]
    \centering
\begin{tikzpicture}[scale=0.9, every node/.style={circle,draw,fill=black!4,inner sep=1pt,minimum size=5pt},line width=0.8pt]

  \foreach \i in {1,...,6} {
    \ifnum\i=1
      \node[fill=red] (v\i) at ({360/6*(\i-1)+30}:1) {$\scriptstyle\mathrm{1}$};
    \else\ifnum\i=2
      \node[fill=red] (v\i) at ({360/6*(\i-1)+30}:1) {$\scriptstyle\mathrm{2}$};
    \else
      \node (v\i) at ({360/6*(\i-1)+30}:1) {$\scriptstyle \mathrm{\i}$};
    \fi\fi
  }
  \foreach \i in {1,...,6} {
    \pgfmathtruncatemacro\nexti{mod(\i,6)+1}
    \draw (v\i) -- (v\nexti);
  }
  \foreach \i in {1,...,6} {
    \pgfmathsetmacro{\ang}{360/6*(\i-1)+30}
    \pgfmathsetmacro{\prev}{\ang+18}
    \pgfmathsetmacro{\next}{\ang-18}
    \def\outdist{1.45}
    \ifnum\i=3
      \node[fill=red] (p\i a) at (\prev:\outdist) {};
      \node (p\i b) at (\next:\outdist) {};
      \node[draw=none, fill=none, left=2pt] at (p\i a) {$b$};
    \else
      \node (p\i a) at (\prev:\outdist) {};
      \node (p\i b) at (\next:\outdist) {};
    \fi
    \draw (v\i) -- (p\i a);
    \draw (v\i) -- (p\i b);
    \draw (p\i a) -- (p\i b);
  }
\end{tikzpicture}
    \caption{The graph $C_6 \odot K_2$}
    \label{Ex_1}
\end{figure}
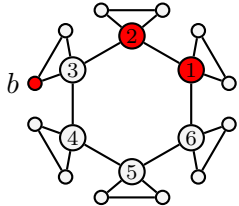
The converse is not true. For example, consider $G = C_6 = (1,2,3,4,5,6)$ and $H = K_2$. Let $A = \{1,2\}$ and let $B \subseteq V(H_3)$. In this case, $B \cap \left(\cup_{a \in A} V(H_a)\right) = \emptyset$. Let $b \in B$. Then the vertices $1$ and $b$ are not $A \cup B$-visible, since every shortest $(1,b)$-path passes through $2 \in A$. Hence, $A \cup B$ is not a mutual-visibility set of $C_6 \odot K_2$.%
\begin{lemma}\label{P2.lem6}
    Let $G$ be a graph with $|V(G)| \geq 2$. If $M$ is a mutual-visibility set of $G \odot H$ and $h_u, h_v \in M$, where $h_u \in V(H_u) , h_v\in V(H_v) , u \neq v$, then $\{u,v\} \cap M=\emptyset$.
\end{lemma}
\begin{proof}
    Suppose $\{u,v\} \cap M \neq \emptyset$. Without loss of generality, we assume that $u \in M$. Then any shortest path from $h_u$ to $h_v$ contains $u$, which is a contradiction to the mutual-visibility of $M$.
\end{proof}

\begin{defn}
Let $u, v \in V(G)$. A vertex $g \in V(G)\setminus\{u, v\}$ is said to be a \emph{shortest-separator} with respect to $u$ and $v$ if every shortest $(u,v)$-path contains $g$. The collection of all shortest-separators is called the \emph{path-cut} of $G$, denoted by $p_c(G)$.

Let $A$ and $B$ be two disjoint subsets of $V(G)$. A vertex $g$ is said to be a \emph{set-separator} with respect to $A$ and $B$ if $g$ is a shortest-separator for every $u \in A$ and $v \in B$.
\end{defn}
The set of all cut vertices of a graph $G$ is a subset of $p_c(G)$. In the case of cycle graph, $p_c(C_n)= V(C_n)$ for $n \geq 5$.
\begin{lemma}\label{P2.lem7}
 Let $g$ be a set-separator with respect to the disjoint sets $A$ and $B$ of $V(G)$. If $S$ is a mutual-visibility set of $G$, containing $g$, then either $S \cap A =\emptyset$ or $S \cap B =\emptyset$. That is, either $S\subseteq \overline{A}$  or $S\subseteq \overline{B}$. 
\end{lemma}
\begin{proof}
    Suppose $S$ contains elements $a \in A$ and $b \in B$. Since $g$ is a set-separator with respect to $A$ and $B$, every shortest $(a,b)$-path contains $g \in S$. This contradicts mutual-visibility of $S$.
\end{proof}
In Lemma \ref{P2.lem3}, we characterized the mutual-visibility sets of $G\odot H$ that contain a single vertex from $G$ and at least one vertex from $\cup_{w \in V(G)}H_w$. The following lemma addresses the case when a mutual-visibility set contains more than one vertex from $G$.
\begin{lemma}\label{P2.lem8}
Let $G$ and $H$ be two graphs. If $S$ is a mutual-visibility set of $G \odot H$ and $|S \cap V(G)|$ contains at least two vertices, then $S \subseteq (S \cap V(G)) \cup \left( \cup_{w \in V(G)\setminus S} V(H_w) \right)$. In particular, if $|V(H)| \ge 2$, then $|S| < |V(G)|\,|V(H)|$. Moreover, $S \cap V(G)$ is a mutual-visibility set of $G$.
\end{lemma}
\begin{proof}
Let $S$ be a mutual-visibility set of $G \odot H$ and let $S \cap V(G) = \{g_1, \ldots, g_k\}$ with $k \geq 2$. Note that $g_1$ is a set-separator with respect to $ (V(G) \setminus \{g_1\}) \cup \left( \cup_{w \in V(G)\setminus \{g_1\}} V(H_w) \right)$ and $V(H_{g_1})$. By Lemma~\ref{P2.lem7}, either $S$ is contained in the complement of $( V(G) \setminus \{g_1\}) \cup \left( \cup_{w \in V(G)\setminus \{g_1\}} V(H_w) \right)$ or in the complement of $V(H_{g_1})$. That is, $S\subseteq \{g_1\} \cup V(H_{g_1})$ or $S \subseteq V(G) \cup \left( \cup_{w \in V(G)\setminus \{g_1\}} V(H_w) \right)$.  Since $S \cap V(G) = \{g_1, \ldots, g_k\}$ it follows that either   $S\subseteq \{g_1\} \cup V(H_{g_1})$ or $S \subseteq \{g_1, \ldots, g_k\} \cup \left( \cup_{w \in V(G)\setminus \{g_1\}} V(H_w) \right)$.  Since $S \supseteq \{g_1, \ldots, g_k\}$, the first case is not possible, and thus the second case must hold. That is,
\begin{equation}\label{P2.eq1}
    S \subseteq \{g_1, \ldots, g_k\} \cup \left( \cup_{w \in V(G)\setminus \{g_1\}} V(H_w) \right)
\end{equation}

Next, observe that $g_2$ is a set-separator with respect to $(V(G) \setminus \{g_2\}) \cup \left( \cup_{w \in V(G)\setminus \{g_1, g_2\}} H_w \right)$ and $H_{g_2}$. Again, by Lemma~\ref{P2.lem7} and the inclusion relation~\eqref{P2.eq1}, it follows that either  $S \subseteq \{g_2\} \cup H_{g_1} \cup H_{g_2}$ or $S \subseteq \{g_1, \ldots, g_k\} \cup \left( \cup_{w \in V(G)\setminus \{g_1, g_2\}} H_w \right)$. In fact, the second case must hold. Proceeding in this manner for each $g_i$, we eventually obtain $S \subseteq \{g_1, \ldots, g_k\} \cup \left( \cup_{w \in V(G)\setminus \{g_1, \ldots, g_k\}} V(H_w) \right)$. Hence, $S \subseteq (S \cap V(G)) \cup \left( \cup_{w \in V(G)\setminus S} V(H_w) \right)$, since $V(G)\setminus \{g_1, \ldots, g_k\} = V(G) \setminus S$.

Consider, $|\{g_1, \ldots, g_k\} \cup \left( \cup_{w \in V(G)\setminus \{g_1, \ldots, g_k\}} V(H_w) \right)|= k+(|V(G)|-k)|V(H)|= |V(G)||V(H)|-k(|V(H)|-1) < |V(G)||V(H)| $, when $|V(H)|\geq 2$. Therefore, $|S| < |V(G)||V(H)|$. Moreover, $S \cap V(G)$ is a mutual-visibility set of $G \odot H$, since it is a subset of the mutual-visibility set $S$. By Lemma \ref{P2.lem1}, $S \cap V(G)$ is a mutual-visibility set of $G$.
\end{proof}
\begin{theorem}\label{P2.th1}
    Let $G$ and $H$ be two graphs, each with at least two vertices. Then $\mu(G\odot H)=|V(G)||V(H)|$.
\end{theorem}
\begin{proof}
By Lemma~\ref{P2.lem2}, $\mu(G \odot H) \ge |V(G)||V(H)|$, since 
$M = \cup_{v \in V(G)} V(H_v)$ is a mutual-visibility set of $G \odot H$. 
By Lemma~\ref{P2.lem6}, any proper superset of $M$ is not a mutual-visibility set. Suppose $S \subseteq V(G \odot H)$ with $M \nsubseteq S$ and $|S| > |V(G)||V(H)|$. 
Let $r = |M \setminus S| \ge 1$. Then $S$ contains at least $r+1$ vertices of $G$, that is, $|S \cap V(G)| \ge r+1 \ge 2$. If, for every $g \in S \cap V(G)$, at least one vertex of $V(H_g)$ is omitted from $S$, it follows that $|M \setminus S| \ge |S \cap V(G)| \ge r+1$, a contradiction. Thus, there exists $g \in S \cap V(G)$ such that $V(H_g) \subseteq S$. Therefore, $S \nsubseteq (S \cap V(G)) \cup \left( \cup_{w \in V(G)\setminus S} V(H_w) \right)$, and by Lemma~\ref{P2.lem8}, $S$ is not a mutual-visibility set. 

Therefore, the maximum size of a mutual-visibility set of $G \odot H$ is $|V(G)||V(H)|$. This completes the proof.
\end{proof}

\begin{theorem}
  Let $G$ and $H$ be two graphs, each with at least two vertices, then the visibility polynomial of $G \odot H$ is a monic polynomial.    
\end{theorem}
\begin{proof}
 In Theorem \ref{P2.th1}, we proved that $M=\cup_{v \in V(G)} V(H_v)$ is a mutual-visibility set of $G \odot H$ of size $|V(G)| |V(H)| = \mu (G \odot H)$. Let $S$ be a mutual-visibility set of $G \odot H$ other than $M$. If $S \cap V(G) = \emptyset$ then $|S| < |M|= \mu (G \odot H)$. Now consider the case $S \cap V(G) \neq \emptyset$. If $S \cap V(G)=\{v\}$, then by Lemma \ref{P2.lem3}, $|S| \leq 1+|V(H)|$ or $|S| \leq 1+ |\cup_{w \in V(G)\setminus \{v\} } V(H_w)|=1+ (|V(G)|-1)|V(H)|$. In both cases $|S| < |V(G)| |V(H)|$. If  $|S \cap V(G)| \geq 2$, then by Lemma \ref{P2.lem8}, $|S| < |V(G)| |V(H)|$. Therefore, $G \odot H$ has only one mutual-visibility set of size $\mu (G \odot H)$. Hence, the result.  
\end{proof}
In \cite{Stefano}, G.~D. Stefano characterized the graphs satisfying $\mu(G) = |V|$ as follows:
\begin{lemma}[\cite{Stefano}]\label{P2.lem10}
Let $G=(V,E)$ be a graph such that $|V|=n$. Then $\mu(G)=|V|$ if and only if $G\cong K_n$.
\end{lemma}
\begin{theorem}
     If $G= K_1$ and $H$ is a graph of order $m$, then  $$\mu(G\odot H)=\left\{
  \begin{array}{ll}
  m+1 & \text{if } H \text{is complete} \\[5pt]
 m  & \text{if } H \text{is non-complete}
  \end{array}
\right.$$
\end{theorem}
\begin{proof}
If $H$ is a complete graph, then $G \odot H$ is also a complete graph, and $V(G \cup H)$ becomes a mutual-visibility set of $G \odot H$ by Lemma~\ref{P2.lem10}. If $H$ is a non-complete graph and $V(G) = \{v_1\}$, then $V(H)$ is a mutual-visibility set of $G \odot H$ since, for any non-adjacent vertices $u, v \in V(H)$, the path $P = (u, v_1, v)$ is a shortest path from $u$ to $v$ with $V(P) \cap V(H) = \{u, v\}$. Moreover, $V(H) \cup V(G)$ is not a mutual-visibility set of $G \odot H$ since every shortest path from $u$ to $v$ contains $v_1$. Hence, the result follows.
\end{proof}
\begin{defn}
Let $Q \subseteq V(G)$. A subset $W$ of $\overline{Q} = V(G) \setminus Q$ is said to be co-visible with respect to $Q$ if $W$ is $Q$-visible and $\{u, w\}$ is $Q$-visible for all $u \in Q$ and $w \in W$. We abbreviate a co-visible set with respect to $Q$ as a $c_Q$-visible set. A $c_Q$-visible set is said to be maximal if it is not a proper subset of any larger $c_Q$-visible set.

Furthermore, a $c_Q$-visible set $W$ is called an absolute $c_Q$-visible set of $G$ if $Q$ is also a mutual-visibility set of $G$. That is, $W$ is an absolute $c_Q$-visible set if $Q \cup W$ is $Q$-visible. A maximal absolute $c_Q$-visible set is denoted by $\Omega_Q(G)$.
\end{defn}
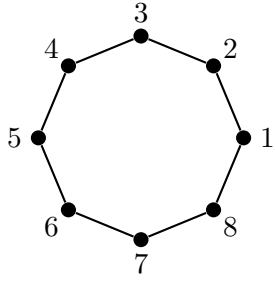
\begin{figure}[h]
    \centering
\begin{tikzpicture}
\begin{scope}[rotate=0, scale=0.9]
    \def\n{8}
    \def\radius{1.5}
    \def\labelradius{1.85} 
    \foreach \i in {1,...,\n} {    
      \node[circle, fill=black, inner sep=2pt] (R\i) at ({360/\n * (\i -1)}:\radius) {};
      \node at ({360/\n * (\i -1)}:\labelradius) {\i};
    }
    \foreach \i in {1,...,\n} {
      \pgfmathtruncatemacro\nexti{mod(\i,\n) + 1}
      \draw[line width=0.8pt] (R\i) -- (R\nexti);
    }
\end{scope}
\end{tikzpicture}
 \caption{The cycle $C_8$}
    \label{Omega_q}
\end{figure}
Note that $\Omega_Q(G)$ need not be unique. For example, consider the cycle graph $C_8$ as depicted in Figure \ref{Omega_q}. If $Q=\{1\}$, then the maximal absolute $c_Q$-visible sets are $\{2,3,4,5,6\}, \{3,4,5,6,7\}$ and $\{4,5,6,7,8\}$. Furthermore, for a fixed $Q$, the maximal absolute $c_Q$-visible sets ($\Omega_Q(G)$) may have different cardinality. For example, let $Q=\{1, 3\}$. The maximal absolute $c_Q$-visible sets are $\{2\}$ and $\{5, 6, 7\}$. If $Q$ has more than two vertices, then there is no $c_Q$-visible set.

The contrapositive of Lemma~\ref{P2.lem5} can be stated as follows: Let $A \subseteq V(G)$ with $|A| \geq 2$ and $B \subseteq \cup_{w \in V(G)} V(H_w)$. If $A \cup B$ is a mutual-visibility set of $G \odot H$, then $B \cap \left(\cup_{a \in A} V(H_a)\right) = \emptyset$, or equivalently, $B \subseteq \cup_{w \in \overline{A}} V(H_w)$. The following lemma characterizes such mutual-visibility sets.
\begin{lemma}\label{P2.lam9}
  Let $Q$ be a proper subset of $V(G)$, $\overline{Q}= V(G)\setminus Q$ and $ \emptyset \neq S_{\overline{Q}} \subseteq \cup_{w \in \overline{Q} } V(H_w)$. A set $S=Q \cup S_{\overline{Q}}$ is a mutual-visibility set of $G \odot H$ if and only if $W=\{w \in \overline{Q} \big/ S \cap V(H_w) \neq \emptyset\}$ is a non-empty absolute $c_Q$-visible set of $G$.
\end{lemma}
\begin{proof}
    $( \Rightarrow)$ Suppose $S=Q \cup S_{\overline{Q}}$ is a mutual-visibility set of $G\odot H$. Since $Q \subseteq S$, it is also a mutual-visibility set of $G\odot H$. Then by Lemma \ref{P2.lem1}, $Q$ is a mutual-visibility set of $G$.

      Let $u, v \in W$, $h_u \in S \cap V(H_u)$ and $h_v \in S \cap V(H_v)$. Since $S$ is a mutual-visibility set of $G\odot H$, $\{h_u, h_v\}$ is $S$-visible. Therefore, there exists a shortest path of the form $P=(h_u,u,\ldots, v, h_v)$ from $h_u$ to $h_v$ such that $V(P) \cap S \subseteq \{h_u, h_v\}$. Let $P'=P \setminus \{h_u, h_v\}$. Then $P'$ is a shortest $(u,v)$-path in $G$ such that $V(P') \cap S = \emptyset$. Consequently, $V(P') \cap Q = \emptyset \subset \{u, v\}$, and hence $W$ is $Q$-visible.

      Let $u \in Q$ and $v \in W$. Since $S \cap V(H_v) \neq \emptyset$, there exist $h_v \in V(H_v)$ such that $u$ and $h_v$ are $S$-visible. Let $P$ be a shortest $(u, h_v)$-path such that $V(P) \cap S \subseteq \{u, h_v\}$. Then $P'=P\setminus \{h_v\}$ is a shortest $(u, v)$-path in $G$ such that $V(P') \cap S \subseteq \{u\}$. Consequently, $V(P') \cap Q \subseteq \{u\} \subset \{u,v\}$. Therefore, $u$ and $v$ are $Q$-visible for all $u \in Q$ and $v \in W$. Hence $W$ is an absolute $c_Q$-visible set of $G$.

$( \Leftarrow)$ If $W=\{w \in \overline{Q} \big/ S \cap V(H_w) \neq \emptyset\}$ is a non-empty absolute $c_Q$-visible set of $G$, then  $Q$ is a mutual-visibility set of $G$. Let $u, v \in Q$. Then there exists a shortest $(u, v)$-path $P$ in $G$ such that $V(P) \cap Q \subseteq \{u, v\}$. Consequently, $V(P) \cap S \subseteq \{u, v\}$, since $V(P) \cap \left(\cup_{w \in \overline{Q}} V(H_w)\right) \subseteq \{u, v\}$. Therefore, $Q$ is a $S$-visible set.

Let $u, v \in S_{\overline{Q}}$. Then $u,v \in V(H_w)$ for some $w \in W$ or $u \in V(H_{w_1})$ and $v \in V(H_{w_2})$ for $w_1, w_2 \in W$. In the first case, $P = (u, v)$ or $P = (u, w, v)$ is a shortest $(u, v)$-path, and $V(P) \cap S = \{u, v\}$. In the latter case, $w_1$ and $w_2$ are $Q$-visible. Then there exists a shortest $(w_1, w_2)$-path $P$ in $G$ such that $V(P) \cap Q = \emptyset$. Since $P$ is a path in $G$, we have $V(P) \cap \left( \cup_{w \in \overline{Q}} V(H_w) \right) = \emptyset$. Therefore, $V(P) \cap S = \emptyset$, where $S = Q \cup S_{\overline{Q}}$. Let $P' = (u, P, v)$ denote the path obtained by adding a pendant vertex $u$ to one endpoint of $P$ and a pendant vertex $v$ to the other. Then $P'$ is a shortest $(u, v)$-path such that $V(P') \cap S = \{u, v\}$. It follows that $S_{\overline{Q}}$ is $S$-visible.

Let $u \in Q$ and $v \in S_{\overline{Q}}$. Then there exists $w \in W$ such that $v \in V(H_w)$. Since $u$ and $w$ are $Q$-visible in $G$, there exists a shortest $(u, w)$-path $P$ in $G$ such that $V(P) \cap Q = \{u\}$. Let $P' = (P, v)$ denote the path obtained by adding a pendant vertex $v$ to $P$ at its endpoint. Then $P'$ is a shortest $(u, v)$-path such that $V(P') \cap S = \{u, v\}$. It follows that $S$ is a mutual-visibility set of $G \odot H$.
\end{proof}
\begin{lemma}\label{P2.lem13}
Let $G = C_n$ with $n \geq 3$, and let $ Q \subseteq V(G)$. A $c_Q$-visible set of $C_n$ exists if and only if $|Q| \leq 2$.
\end{lemma}
\begin{proof}
If $Q = \emptyset$, then $\overline{Q} = V(G)$ is trivially a $c_Q$-visible set. Suppose $|Q| = 1$. In this case, any singleton subset of $\overline{Q}$ is a $c_Q$-visible set. Now let $Q = \{g_1, g_2\}$. If $g_1$ and $g_2$ are not adjacent, then any singleton set containing an intermediate vertex on a shortest $(g_1, g_2)$-path is a $c_Q$-visible set. If $g_1$ and $g_2$ are adjacent, represent the cycle  $C_n$ as $ (g_1, g_2, \ldots, g_n)$. If $n$ is odd, then $\{g_{(n+3)/2}\}$ is a $c_Q$-visible set and if $n$ is even, then $\{g_{(n/2)+1}\}$ is a $c_Q$-visible set.

Now suppose $|Q| \geq 3$. Let $x \in \overline{Q}$. Then $x$ lies on exactly one path $P$ between two vertices in $Q$ such that $P$ does not contain any other vertex of $Q$. Let $y \in V(\overline{P}) \cap Q$; such a $y$ exists since $|Q| \geq 3$. Every shortest $(x,y)$-path contains one of the endpoints of $P$. Therefore, no subset of $\overline{Q}$ is $Q$-visible, and hence no $c_Q$-visible set exists when $|Q| \geq 3$. This completes the proof.
\end{proof}
\begin{prop}\label{P2.prop1}
For $n \geq 3$, let $A \subseteq V(C_n)$ with $|A| \geq 2$, and let $\emptyset \neq B \subseteq \cup_{w \in V(C_n)} V(H_w)$.  
If $S = A \cup B$ is a mutual-visibility set of $C_n \odot H$, then $|A| = 2$.
\end{prop}
\begin{proof}
Assume that $S = A \cup B$ is a mutual-visibility set of $C_n \odot H$. Then, by the contrapositive of Lemma~\ref{P2.lem5}, it follows that $B \subseteq \cup_{w \in \overline{A}} V(H_w)$. By Lemma~\ref{P2.lam9}, the set $
W = \{w \in \overline{A} \mid S \cap V(H_w) \neq \emptyset\}$
is a non-empty absolute $c_A$-visible set of $C_n$. By Lemma~\ref{P2.lem13}, this is possible only if $|A| \leq 2$. Hence, $|A| = 2$.
\end{proof}
By combining Lemma~\ref{P2.lem3} and Proposition~\ref{P2.prop1}, we conclude the following result: If $A \cup B$ is a mutual-visibility set of $C_n \odot H$, where $A \subseteq V(C_n)$ and $B \subseteq \cup_{w \in V(C_n)} V(H_w)$, then $|A| \leq 2$. 
\begin{prop}\label{P2.prop5}
Let $G=C_n$ with $n\ge 3$, and $Q=\{a,b\}\subseteq V(G)$. 
If $a$ and $b$ are adjacent, then $Q$ has a unique maximal absolute $c_Q$-visible set.  If $a$ and $b$ are nonadjacent, then any two maximal absolute $c_Q$-visible sets are disjoint.
\end{prop}
\begin{proof}
Let $P$ and $P'$ be the two $(a,b)$-paths in $C_n$, and set $k=|E(P)|$, $l=|E(P')|$, so $k+l=n$. Without loss of generality we assume that, $k\le l$ and let $A =V(P)\setminus\{a,b\}$ and  $B=V(P')\setminus\{a,b\}$. Note that $C_n\setminus Q$ has exactly the two components induced by $A$ and $B$. Since $|Q|=2$, $Q$ is a mutual-visibility set of $C_n$.  Thus, the adjective “absolute” is automatically satisfied for any $c_Q$-visible set.

\medskip
\noindent\textbf{Separation}: If $x\in A$ and $y\in B$, then every $(x,y)$-path meets $Q$ internally; hence $x$ and $y$ are not $Q$-visible. 
Therefore, any $c_Q$-visible set $W$ must be contained either in  $A$ or in $B$.

For $w\in B$, let $d_{P'}(a,w)$ be the length of the $P'$-subpath from $a$ to $w$. Then $d_{P'}(a,w)+d_{P'}(b,w)=l$, and the two $(a,w)$-paths in $C_n$ have lengths $d_{P'}(a,w)$ (along $P'$) and $k+d_{P'}(b,w)$ (via $P$ then along $P'$). Hence
\[
\text{$a$ and $w$ are $Q$-visible} \iff d_{P'}(a,w)\le k+d_{P'}(b,w).
\]
By symmetry, $b$ and $w$ are $Q$-visible if and only if $d_{P'}(b,w)\le k+d_{P'}(a,w)$ . Therefore, $w\in B$ is $Q$-visible if and only if
\begin{equation}\label{P2.eq2}
|d_{P'}(a,w)- d_{P'}(b,w)| \leq k 
\end{equation}

\medskip
\noindent\textbf{Case 1}: Suppose that $a$ and $b$ are adjacent ($k=1$).
Here $A=\emptyset$ and $B$ is the long $(a,b)$-path of length $l=n-1$. By the inequality \eqref{P2.eq2}, a vertex $w \in B$ is $Q$-visible only when $|d_{P'}(a,w)- d_{P'}(b,w)|\leq 1$, that is, the middle vertices of $B$. If $n$ is odd, there is only one such vertex, and two consecutive vertices if $n$ is even. In both cases, they are $Q$-visible. No other vertex of $B$ is $Q$-visible, so this set is the only possible nonempty $c_Q$-visible set; in particular, it is the unique maximal absolute $c_Q$-visible set.

\medskip
\noindent\textbf{Case 2}: Suppose that $a$ and $b$ are nonadjacent ($k\geq 2$). Then, $A$ itself is a $c_Q$-visible set: for $x, y\in A$ the unique shortest $(x,y)$-path is the subpath of $P$, and for $u\in Q$, $w\in A$ the $(u,w)$-subpath of $P$ is a shortest $(u,w)$-path. Hence, the internal vertices of all these paths avoid $Q$. Thus by separation remark, $A$ is a maximal absolute $c_Q$-visible set. 

Next we define $ B^\star =\{\,w\in B:\ |d_{P'}(a,w)- d_{P'}(b,w)|\leq k\,\}$. By the inequality \eqref{P2.eq2}, every $w\in B^\star$ is $Q$-visible. Let $x,y\in B^\star$ with $d_{P'}(a,x) < d_{P'}(a,y)$. Then the distance between $x$ and $y$ along $P'$ is $d_{P'}(a,y)-d_{P'}(a,x)$. Substituting $d_{P'}(b,w)= l- d_{P'}(a,w)$ in the inequality \eqref{P2.eq2}, we get $|2d_{P'}(a,w)- l|\leq k\,$. That is $d_{P'}(a,w)\in\big[\frac{l-k}{2},\ \frac{l+k}{2}\big]$ and,
\[
d_{P'}(x,y)=d_{P'}(a,y)-d_{P'}(a,x) \leq \frac{l+k}{2} - \frac{l-k}{2} = k \leq \frac{n}{2}.
\]
Hence $(x,y)$-path along $P'$ has length at most $\frac{n}{2}$, so it is the only one shortest $(x,y)$-path in $C_n$. Moreover, all its internal vertices are disjoint from $Q$. Therefore, $B^\star$ is $Q$-visible, and hence it is an absolute $c_Q$-visible set.

Maximality and uniqueness inside $B$ directly follow from the inequality \eqref{P2.eq2}: if $w\in B\setminus B^\star$, then at least one of pair $\{a, w\}$ or $\{b,w\}$ has all shortest paths passing through the other vertex of $Q$, so $w$ cannot belong to any $c_Q$-visible set. Thus $B^\star$ is the unique maximal $c_Q$-visible subset of $C_n$ in $B$. Finally, by the Separation remark, the maximal absolute $c_Q$-visible sets $A$ and $B^\star$ are disjoint. This completes the proof.
\end{proof}
\begin{prop}\label{P2.prop3}
    Let $\emptyset \neq Q \subseteq V(P_n)$, where $n \geq 4$. If $Q$ has more than one maximal absolute $c_Q$-visible set, then they are disjoint.
\end{prop}
\begin{proof}
In \cite{sandi}, B.~Csilla et al. showed that the visibility polynomial of a path graph of order $n \geq 2$ is given by $\mathcal{V}(P_n) = 1 + nx + \binom{n}{2}x^2
$. Therefore, $Q$ has no absolute $c_Q$-visible sets when $|Q| \geq 3$, and all subsets of size at most two are mutual-visibility sets of $P_n$. Let $P_n = (1,2,\ldots,n)$. If $Q = \{1\}$ or $Q = \{n\}$, then the maximal absolute $c_Q$-visible set is $\{2,3,\ldots,n\}$ or $\{1,2,\ldots,n-1\}$, respectively. If $Q = \{k\}$, where $1 < k < n$, then the maximal absolute $c_Q$-visible sets are $\{1,\ldots,k-1\}$ and $\{k+1,\ldots,n\}$.

Now, let $Q = \{a,b\}$. If $a$ and $b$ are adjacent in $P_n = (1,2,\ldots,n)$ and $a <b$, consider any $p \in \overline{Q}$. Then either $p < a$ or $p > b$. In the first case, the shortest $(p,b)$-path contains $a$, and in the latter case, the shortest $(a,p)$-path contains $b$. Therefore, no $c_Q$-visible set exists in this case. If $a$ and $b$ are not adjacent, then subsets of $\{a+1,\ldots,b-1\}$ are the only $c_Q$-visible sets. Moreover, if $W \subseteq \overline{Q}$ contains an element outside $\{a+1,\ldots,b-1\}$, say $q$, then either $\{q,b\}$ or $\{a,q\}$ is not $Q$-visible. Hence, $\{a+1,\ldots,b-1\}$ is the unique maximal absolute $c_Q$-visible set. This completes the proof.
\end{proof}
\begin{prop}\label{P2.prop4}
Let $\emptyset \neq Q \subseteq V(K_n)$, where $n \geq 3$. Then there exists a unique maximal absolute $c_Q$-visible set.
\end{prop}
\begin{proof}
We claim that $\overline{Q}$ is a $c_Q$-visible set. Since every two vertices in $\overline{Q}$ are adjacent, they are $Q$-visible. Similarly, any $q \in Q$ and $w \in \overline{Q}$ are $Q$-visible. In~\cite{VP_1}, the present authors proved that every subset of $V(K_n)$ is a mutual-visibility set of $K_n$. Therefore, $\overline{Q}$ is an absolute $c_Q$-visible set. By definition, a $c_Q$-visible set is a subset of $\overline{Q}$; hence $\overline{Q}$ is the unique maximal absolute $c_Q$-visible set.
\end{proof}
In Proposition~\ref{P2.prop5}, we established that if $|Q| = 2$ in a cycle, then the maximal absolute $c_Q$-visible sets are disjoint. However, this property does not hold in general. For example, consider the graph $G$ in Figure~\ref{mul_Omega_q} with $Q = \{4,6\}$. 
In this case, $\{1,2,5\}$ and $\{1,2,3\}$ are both maximal absolute $Q$-visible sets.
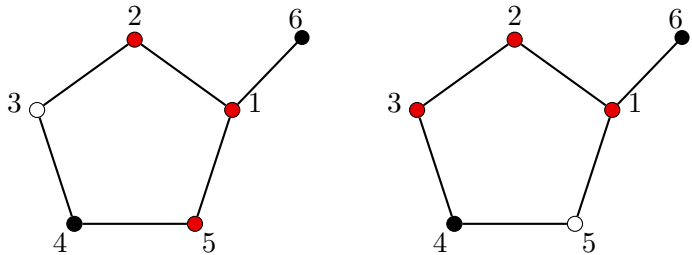
\begin{figure}[h]
    \centering
\begin{tikzpicture}
  \begin{scope}[xshift=5cm, rotate=18, scale=0.9]
    \def\n{5}
    \def\radius{1.5}
    \def\labelradius{1.85} 
    \foreach \i in {1,...,\n} {
      \ifnum\i=1
        \def\mycolor{black!10!red}
      \else\ifnum\i=2
        \def\mycolor{black!10!red}
      \else\ifnum\i=3
        \def\mycolor{black!10!red}
    \else\ifnum\i=5
        \def\mycolor{none}
      \else
        \def\mycolor{black}
      \fi\fi\fi\fi
    
      \node[circle, draw=black,fill=\mycolor, inner sep=2pt] (R\i) at ({360/\n * (\i -1)}:\radius) {};
      \node at ({360/\n * (\i -1)}:\labelradius) {\i};
    }
    \foreach \i in {1,...,\n} {
      \pgfmathtruncatemacro\nexti{mod(\i,\n) + 1}
      \draw[line width=0.8pt] (R\i) -- (R\nexti);
    }
    \coordinate (R6) at ($(R1) + (1.3,0.7)$);
    \node[circle, fill=black, inner sep=2pt] at (R6) {};
    \node at ($(R6) + (0,0.3)$) {6};
    \draw[line width=0.8pt] (R1) -- (R6);
  \end{scope}
  \begin{scope}[rotate=18,scale=0.9]
  \def\n{5}
    \def\radius{1.5}
    \def\labelradius{1.85} 
    \foreach \i in {1,...,\n} {
      \ifnum\i=1
        \def\mycolor{black!10!red}
      \else\ifnum\i=2
        \def\mycolor{black!10!red}
      \else\ifnum\i=5
        \def\mycolor{black!10!red}
    \else\ifnum\i=3
        \def\mycolor{none}
      \else
        \def\mycolor{black}
      \fi\fi\fi\fi
    
      \node[circle, draw=black,fill=\mycolor, inner sep=2pt] (R\i) at ({360/\n * (\i -1)}:\radius) {};
      \node at ({360/\n * (\i -1)}:\labelradius) {\i};
    }
    \foreach \i in {1,...,\n} {
      \pgfmathtruncatemacro\nexti{mod(\i,\n) + 1}
      \draw[line width=0.8pt] (R\i) -- (R\nexti);
    }
    \coordinate (R6) at ($(R1) + (1.3,0.7)$);
    \node[circle, fill=black, inner sep=2pt] at (R6) {};
    \node at ($(R6) + (0,0.3)$) {6};
    \draw[line width=0.8pt] (R1) -- (R6);
\end{scope}
\end{tikzpicture}
 \caption{The graphs $G$ with intersecting $\Omega_Q(G)$}
    \label{mul_Omega_q}
\end{figure}
\section{Visibility polynomial of the corona product of two graphs} \label{sec4}

Using a computing facility and a Python implementation, we verified that, for a given mutual-visibility set $Q$, the maximal absolute $c_Q$-visible sets are disjoint in all connected graphs of order $3$ and $4$. Furthermore, among the $21$ non-isomorphic connected graphs of order $5$, $18$ satisfy this property, and among the $112$ non-isomorphic connected graphs of order $6$, $73$ satisfy it. In addition, Propositions~\ref{P2.prop3} and~\ref{P2.prop4} guarantee the existence of infinitely many such graphs. This empirical evidence motivates the formal introduction of the following definition and a related result.
\begin{defn}
Let $G$ be a graph and let $ Q \subseteq V(G)$. The subset $Q$ is said to be disjoint-visible if, whenever $Q$ has more than one maximal absolute $c_Q$-visible set, these sets are pairwise disjoint; that is, the sets $\Omega_Q(G)$ are pairwise disjoint. A graph $G$ is called absolute-clear if every non-empty subset of $V(G)$ is disjoint-visible.
\end{defn}
The graph shown in Figure~\ref{mul_Omega_q} is not absolute-clear. However, the graph $G_1$, shown in Figure \ref{abs_clear_prob}, is absolute-clear.
\begin{lemma}\label{P2.lem11}
Let $G$ be an absolute clear graph, $\emptyset \neq Q \subsetneq V(G)$ and $\overline{Q}= V(G)\setminus Q$. Let $S=Q \cup S_{\overline{Q}}$, where $\emptyset \neq S_{\overline{Q}} \subseteq \cup_{w \in \overline{Q} } V(H_w)$, be a mutual-visibility set of $G \odot H$. The contribution to the visibility polynomial of $G \odot H$ corresponding to mutual-visibility sets of the form $S$ is 
$$\sum_{Q} \sum_{\Omega_Q(G)}\left((1+x)^{|\Omega_Q(G)||V(H)|}-1\right)x^{|Q|}$$
where $Q$ is a proper mutual-visibility set of $G$.
\end{lemma}
\begin{proof}
Let $Q$ be a proper subset of $V(G)$ and $S=Q \cup S_{\overline{Q}}$ be a mutual-visibility set of $G \odot H$. By Lemma~\ref{P2.lam9}, it follows that $W = \{w \in \overline{Q} / S \cap V(H_w) \neq \emptyset\}$ is a non-empty absolute $c_Q$-visible set of $G$. Therefore, $Q$ is a proper mutual-visibility set of $G$. Again, by Lemma~\ref{P2.lam9}, the set $S = Q \cup S_{\overline{Q}}$, where $\emptyset \neq S_{\overline{Q}} \subseteq \cup_{w \in \overline{Q}} V(H_w)$, is a mutual-visibility set of $G \odot H$ if and only if $\emptyset \subsetneq S_{\overline{Q}} \subseteq \cup_{w \in \Omega_Q(G)} V(H_w)$, where $\Omega_Q(G)$ is a maximal absolute $c_Q$-visible set corresponding to $Q$. Thus, for every non-empty subset of $\cup_{w \in \Omega_Q(G)} V(H_w)$ of size $p$, where $1 \leq p \leq |\Omega_Q(G)| |V(H)|$, there exists a mutual-visibility set of $G \odot H$ of size $p + |Q|$. Therefore, the contribution to the visibility polynomial $\mathcal{V}(G \odot H)$ corresponding to $\Omega_Q(G)$ is $\left((1 + x)^{|\Omega_Q(G)||V(H)|} - 1\right) x^{|Q|}$. If  $Q \subseteq V(G)$ has more than one maximal $c_Q$-visible set, then they are disjoint, since $G$ is absolute-clear. Hence, the total contribution to $\mathcal{V}(G \odot H)$ corresponding to a mutual-visibility set $Q$ of $G$ is $\sum_{\Omega_Q(G)} \left((1 + x)^{|\Omega_Q(G)||V(H)|} - 1\right) x^{|Q|}$. Since $S_{\overline{Q}} \neq \emptyset$, it follows that $\Omega_Q(G)$ is non-empty. Note that, for a mutual visibility set $Q$, $\Omega_Q(G)$ may be empty, and in that case, the corresponding term in the sum becomes zero. This completes the proof.
\end{proof}
Suppose that the graph $G$ is not absolute-clear. Then there exists a subset $\emptyset \neq Q \subsetneq V(G)$ that is not disjoint-visible. Let $\Gamma_Q(G)$ denote the collection of all maximal absolute $c_Q$-visible sets, that is, $\Gamma_Q(G) = \{\Omega_Q(G)\}$. Since $Q$ is not disjoint-visible, at least two sets in $\Gamma_Q(G)$ have non-empty intersection. Using the principle of inclusion-exclusion, the contribution to the visibility polynomial $\mathcal{V}(G \odot H)$ corresponding to mutual-visibility sets of the form $S = Q \cup S_{\overline{Q}}$, where $\emptyset \neq S_{\overline{Q}} \subseteq \cup_{w \in \overline{Q}} V(H_w)$, is given by
$$
\sum_{\emptyset \neq \mathcal{J} \subseteq \Gamma_Q(G)} (-1)^{|\mathcal{J}|+1} \left((1+x)^{\left| \cap_{W \in \mathcal{J}} W \right||V(H)|}-1\right)x^{|Q|}
$$
Using this result, we can generalize Lemma~\ref{P2.lem11} as follows.

\begin{lemma}\label{P2.lem12}
Let $G$ be a graph, $\emptyset \neq Q \subsetneq V(G)$ and $\overline{Q}= V(G)\setminus Q$. Let $S=Q \cup S_{\overline{Q}}$, where $\emptyset \neq S_{\overline{Q}} \subseteq \cup_{w \in \overline{Q} } V(H_w)$, be a mutual-visibility set of $G \odot H$. The contribution to the visibility polynomial of $G \odot H$ corresponding to mutual-visibility sets of the form $S$ is $\sum_{Q} p_Q(x)$, where the sum is taken over all proper mutual-visibility sets $Q$ of $G$ and 
$$p_Q(x) = \left\{
  \begin{array}{ll}
  \sum_{\Omega_Q(G)}\left((1+x)^{|\Omega_Q(G)||V(H)|}-1\right)x^{|Q|}& \text{if } Q \text{ is disjoint-visible} \\[10pt]
  \sum_{\emptyset \neq \mathcal{J} \subseteq \Gamma_Q(G)} (-1)^{|\mathcal{J}|+1} \left((1+x)^{\left| \cap_{W \in \mathcal{J}} W \right||V(H)|}-1\right)x^{|Q|} & \text{Otherwise}
  \end{array}
\right.$$
\end{lemma}
Consider the corona product of two graphs $G$ and $H$ with $m$ and $n$ vertices, respectively. 
Synthesizing the insights developed in Lemmas~\ref{P2.lem1}, \ref{P2.lem2}, \ref{P2.lem3}, and \ref{P2.lem12}, 
we obtain a complete characterization of mutual-visibility sets in the corona product $G \odot H$. 
The following theorem, which summarizes this analysis, provides an explicit expression for the visibility polynomial of the corona product of two graphs. 
Moreover, the first case of Lemma~\ref{P2.lem3} motivates the following definition.
\begin{defn}
Let $G$ be a graph and $0 \leq d \leq diam(G)$. The diameter-restricted visibility polynomial $\mathcal{V}_d(G)$ is defined by 
$\mathcal{V}_d(G)= \sum_{i \geq 0} r_{i,d} x^i$, where $r_{i,d}$ denotes the number of mutual-visibility sets of order $i$ of $G$ having diameter at most $d$. That is, 
$\mathcal{V}_d(G)= 1+\sum_{i \geq 1}\left[\sum_{j \leq d}\Theta_{i,j}(G)\right]x^i$.
\end{defn}
The diameter of any mutual-visibility set of a graph $G$ is at most the diameter of the graph, and there exist two vertices $a, b \in V(G)$ such that $d_G(a, b) = diam(G)$. Therefore, $\mathcal{V}_d(G)=\mathcal{V}(G)$ if and only if $d=diam(G)$.
\begin{theorem}
Let $G$ be a graph with $m = |V(G)| > 1$ and let $H$ be a graph with $n = |V(H)|$, then $\mathcal{V}(G \odot H)= \mathcal{V}(G)+\left((1+x)^{mn}-1 \right)+mx(\mathcal{V}_2(H)-1) + \sum_{Q}p_Q(x)$, where the sum is taken over all proper mutual-visibility sets $Q$ of $G$, and $p_Q(x)$ is defined as in Lemma~\ref{P2.lem12}.
\end{theorem}
\begin{proof}
All mutual-visibility sets of $G \odot H$ are characterized in Lemmas~\ref{P2.lem1}, \ref{P2.lem2}, \ref{P2.lem3}, and \ref{P2.lem12}. Accordingly, we compute $\mathcal{V}(G \odot H)$ by considering four distinct cases. By Lemma~\ref{P2.lem1}, all mutual-visibility sets of $G$ are also mutual-visibility sets of $G\odot H$. Therefore, the contribution to $\mathcal{V}(G \odot H)$ in this case is $\mathcal{V}(G)$. By Lemma~\ref{P2.lem2}, all subsets of $\cup_{w \in V(G)} V(H_w)$ are also mutual-visibility sets of $G \odot H$; however, the empty subset is already counted in the previous case. Hence, the contribution in this case is $\left((1+x)^{mn} - 1\right)$. Next, we compute the contribution corresponding to Case~1 of Lemma~\ref{P2.lem3}. Specifically, mutual-visibility sets of the form $\{v\} \cup B$, where $B$ is a mutual-visibility set of $H_v$ with diameter at most two, contribute $x(\mathcal{V}_{2}(H)-1)$. Since there are $m$ such vertices $v$ in $V(G)$, the total contribution in this case is $mx(\mathcal{V}_2(H)-1)$.

Finally, Lemma~\ref{P2.lem12} accounts for the remaining contribution to $\mathcal{V}(G \odot H)$, which includes the mutual-visibility sets referred to in the second case of Lemma~\ref{P2.lem3}. Specifically, this corresponds to mutual-visibility sets of the form $S = Q \cup S_{\overline{Q}}$, where $\emptyset \neq Q \subsetneq V(G)$ and $\emptyset \neq S_{\overline{Q}} \subseteq \cup_{w \in \overline{Q}} V(H_w)$. Note that, by Lemma~\ref{P2.lem5}, no proper superset of $V(G)$ is a mutual-visibility set of $G \odot H$. This completes the proof.
\end{proof}
\begin{table}[htbp]
    \centering
    \caption{$Q, \Omega_Q(G)$ and its contribution to $\mathcal{V}(P_3 \odot K_2)$}\label{P2.tbl1}
    \begin{tabularx}{0.9\linewidth}{@{}XXll@{}}
        \toprule
       Q & $\Omega_Q(G)$ & $\sum_{\Omega_Q(G)}\left((1+x)^{|\Omega_Q(G)||V(H)|}-1\right)x^{|Q|}$\\
        \midrule
        $\{1\}$ & $\{2, 3\}$& $((1+x)^4-1)x$\\
        $\{2\}$ & $\{1\}$, $\{3\}$ &$((1+x)^2-1)x+((1+x)^2-1)x$\\
        $\{3\}$ & $\{1,2\}$ &$((1+x)^4-1)x$\\
        $\{1,3\}$ & $\{2\}$ &$((1+x)^2-1)x^2$\\
        $\{1,2\}$ & $\emptyset$ &0\\
        $\{2,3\}$ & $\emptyset$ &0\\
        \bottomrule
    \end{tabularx}
\end{table}
\textbf{Example:} Let $G=P_3=(1,2,3)$ and $H=K_2$. Then 
$\mathcal{V}(P_3 \odot K_2)= \mathcal{V}(P_3)+((1+x)^{6}-1)+3x(\mathcal{V}_2(K_2)-1)+\sum_Q \sum_{\Omega_Q(G)}\left((1+x)^{|\Omega_Q(G)||V(H)|}-1\right)x^{|Q|}$, since $P_3$ is absolute-clear [See Prop. \ref{P2.prop3}]. All possible $Q$'s, maximal absolute $c_Q$-visible sets and corresponding contribution to $\mathcal{V}(P_3 \odot K_2)$ are listed in Table~\ref{P2.tbl1}. $\mathcal{V}(P_3)= 1+3x+3x^2$ from \cite{sandi} and $\mathcal{V}_2(K_2)=\mathcal{V}(K_2)=(1+x)^2$ from \cite{VP_1}. Therefore, $\mathcal{V}(P_3 \odot K_2)= 1+3x+3x^2+((1+x)^{6}-1)+3x((1+x)^2-1)+2((1+x)^4-1)x+2((1+x)^2-1)x+((1+x)^2-1)x^2 = 1 + 9x + 36x^{2} + 39x^{3} + 24x^{4} + 8x^{5} + x^{6}
$
\section{Concluding Remarks}
This paper examined the structure of mutual-visibility sets in the corona product of two connected graphs. By establishing a series of lemmas, we identified necessary and sufficient conditions for a subset of vertices to constitute a mutual-visibility set in the corona graph $G \odot H$. Leveraging these structural insights, we derived a closed-form expression for the visibility polynomial of $G \odot H$, which encodes the enumeration of mutual-visibility sets by their cardinality.

In this work, we introduced the concepts of $c_Q$-visible sets and absolute-clear graphs as tools for expressing the visibility polynomial of a graph in a concise and structured manner. These concepts offer a new perspective in visibility theory, providing both a compact polynomial framework and insight into the structural behavior of graphs under visibility constraints.

\begin{figure}[h]
    \centering
\begin{tikzpicture}
\begin{scope}[xshift=-4.5cm, rotate=18, scale=0.9]
    \def\n{5}
    \def\radius{1.5}
    \def\labelradius{1.85}
    \foreach \i in {1,...,\n} {
      \node[circle, fill=black, inner sep=2pt] (B\i) at ({360/\n * (\i -1)}:\radius) {};
      \node at ({360/\n * (\i -1)}:\labelradius) {\i};
    }
    \draw[line width=0.8pt] (B2) -- (B3);
    \draw[line width=0.8pt] (B3) -- (B4);
    \draw[line width=0.8pt] (B4) -- node[pos=0.5, yshift=-5mm] {$G_1$}(B5);
    \draw[line width=0.8pt] (B5) -- (B1);
    \draw[line width=0.8pt] (B1) -- (B3);
    \draw[line width=0.8pt] (B3) -- (B5);
\end{scope}

\begin{scope}[xshift=0cm, rotate=18, scale=0.9]
    \def\n{5}
    \def\radius{1.5}
    \def\labelradius{1.85}
    \foreach \i in {1,...,\n} {
      \node[circle, fill=black, inner sep=2pt] (B\i) at ({360/\n * (\i -1)}:\radius) {};
      \node at ({360/\n * (\i -1)}:\labelradius) {\i};
    }
    \draw[line width=0.8pt] (B2) -- (B3);
    \draw[line width=0.8pt] (B3) -- (B4);
    \draw[line width=0.8pt] (B4) -- node[pos=0.5, yshift=-5mm] {$G_2$}(B5);
    \draw[line width=0.8pt] (B5) -- (B1);
    \draw[line width=0.8pt] (B1) -- (B3);
    \draw[line width=0.8pt] (B3) -- (B5);
    \draw[line width=0.8pt] (B1) -- (B4);
    \draw[line width=0.8pt, red] (B1) -- (B4);
\end{scope}

\begin{scope}[xshift=4.5cm, rotate=18, scale=0.9]
    \def\n{5}
    \def\radius{1.5}
    \def\labelradius{1.85}
    \foreach \i in {1,...,\n} {
      \node[circle, fill=black, inner sep=2pt] (B\i) at ({360/\n * (\i -1)}:\radius) {};
      \node at ({360/\n * (\i -1)}:\labelradius) {\i};
    }
    \foreach \i in {1,...,\n} {
      \pgfmathtruncatemacro\nexti{mod(\i,\n) + 1}
      \draw[line width=0.8pt] (B\i) -- (B\nexti);
    }
    \draw[line width=0.8pt,red] (B1) -- (B2);
    \draw[line width=0.8pt] (B1) -- (B3);
    \draw[line width=0.8pt] (B3) -- (B5);
    \node[draw=none, fill=none] at (-0.5,-1.7) {$G_3$};
\end{scope}
\end{tikzpicture}
    \caption{Problem on absolute-clear graphs}
    \label{abs_clear_prob}
\end{figure}
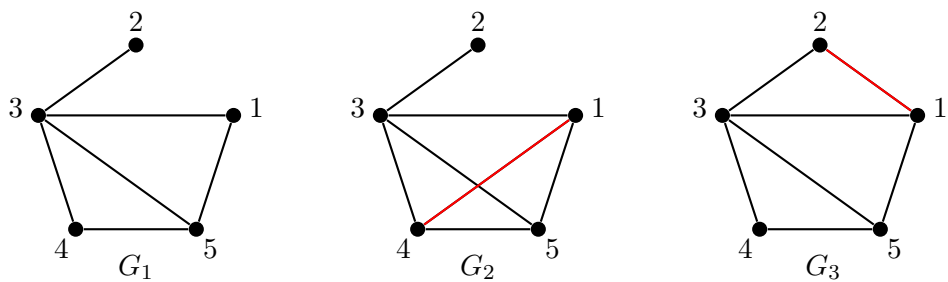

As an illustrative case, Figure~\ref{abs_clear_prob} presents the graph $G_ 
1$, which is absolute-clear. Adding an edge from vertex 1 to vertex 4 yields $G_2$, which retains the absolute-clear property. In contrast, adding an edge from vertex 1 to vertex 2 produces $G_3$, which is no longer absolute-clear. This observation motivates the following question:\\
Problem: What are the necessary and sufficient conditions for adding an edge to a graph such that it remains absolute-clear?

Future research directions include developing precise structural criteria for edge additions that preserve the absolute-clear property, and extending the analysis to other graph operations such as edge deletion, subdivision, and various graph products to understand their influence on $C_Q$-visible sets and absolute-clear graphs. The relationship between the visibility polynomial and other established graph polynomials may be investigated, potentially revealing deeper combinatorial connections. Furthermore, these concepts can be applied to practical domains like network design, robotics, and sensor placement, where visibility and structural clarity are crucial.

\textbf{Data availability:} All non-isomorphic graphs of order up to nine were generated using the nauty27r3 software package.
\bibliographystyle{plainurl}
\bibliography{cas-refs}
\end{document}